\documentclass[12pt]{article} 
\newcommand{\be}{\begin{equation}}
\newcommand{\ee}{\end{equation}}
\newcommand{\ba}{\begin{array}}
\newcommand{\ea}{\end{array}}

\newcommand{\bea}{\begin{eqnarray*}}
\newcommand{\eea}{\end{eqnarray*}}
\newcommand{\bean}{\begin{eqnarray}}
\newcommand{\eean}{\end{eqnarray}}

\def\ds{\displaystyle}

\newtheorem{lemma}{Lemma}[section]
\newtheorem{remark}{Remark}[section]

\newtheorem{theorem}{Theorem}[section]
\newtheorem{corollary}{Corollary}[section]
\newtheorem{proposition}{Proposition}[section]
\usepackage{amsmath, amssymb}
\usepackage{latexsym}
\newcommand{\R}{\mathbb{R}}

\def\Box{\leavevmode\vbox{\hrule
     \hbox{\vrule\kern5pt\vbox{\kern5pt}%
           \vrule}\hrule}}
\renewcommand{\square}{\hfill$\Box$}
\begin{document}
\title{Determination of small linear perturbations in the diffusion coefficient from partial dynamic boundary measurements}

\author{Aymen Jbalia\thanks{Department of Mathematics, Faculty of Sciences, 7021 Zarzouna, Bizerte, Tunisia.
(Email: jbalia.aymen@yahoo.fr)} \and \ Abdessatar Khelifi \thanks{ Faculty of
Sciences of Bizerte, University of Carthage
(Email: abdessatar.khelifi@fsb.rnu.tn) }}



\maketitle \abstract{This work deals with an inverse boundary value problem arising
from the equation of heat conduction. We reconstruct small
perturbations of the (isotropic) heat conductivity distribution from partial (on accessible part of the boundary) dynamic boundary
measurements and for finite interval in
time.
By constructing of appropriate test functions, using a control method, we provide a rigorous derivation of the
inverse Fourier transform of the perturbations in the diffusion coefficient as the leading order of an appropriate averaging of
the partial dynamic boundary measurements.}\\


\noindent {\footnotesize {\bf 2010 Mathematics Subject
Classifications}: 35K05, 35R30, 80A23, 35B40.

\section{Introduction}
This paper is devoted to the identification of small amplitude perturbations, in the smooth diffusion coefficient for the heat equation, using partial boundary measurements.\\
The inverse heat conduction problem arises in most thermal manufacturing processes
of solids and has recently attracted much attention. In this inverse problem,
for the heat equation, one is requested to reconstruct a heat conductivity distribution by measuring on an accessible part of the boundary. Indeed, we exhibit appropriate boundary measurements by using exact boundary control data to reconstruct $\emph{\textbf{c}}_\alpha(x)$ approximately, provided it deviates only slightly from known constant $\emph{\textbf{c}}_0$.\\
Notice that reconstruction methods that allow partial boundary data are very interesting because, in most experimental
settings, one does not have access to measurements on the whole boundary.\\
The problems to be discussed in this article generalize the approaches elaborated by Somersalo, Isaacson and Cheney in \cite{SIC}, Ammari in \cite{Ammari}, Darbas and Lorhengel in \cite{Darbas1} from reconstructions of electromagnetic parameters to an inverse problem of reconstructing an unknown coefficient in a parabolic equation.\\
Following the approaches used in \cite{Ammari,Darbas1}, our reconstruction method based on the knowledge of the boundary measurements requires the resolution of an exact boundary controllability problem by using the Hilbert Uniqueness Method (HUM) \cite{Lions}. Unlike electromagnetic waves studied in the above references, the exact boundary control problem for the heat equation is ill-posed in general. Thanks to Carthel, Glowinski and Lions \cite{Glowinski1}, this ill-posedness is surmounted by using specific regularization procedures.\\
To the best of our knowledge, the present paper is the first attempt to design an effective method to determine a coefficient valued-function which quantifies the perturbations of the thermal conductivity with respect to the homogeneous background medium, and satisfies some specific conditions.\\
For the stationary case, the inverse conductivity problem has been studied by several authors through different approaches. Nachman \cite{Nachman} proved an uniqueness result for the
diffusion coefficient $\ds c\in \mathcal{C}^2(\bar{\Omega})$ and Astala, P$\ddot{a}$iv$\ddot{a}$rinta \cite{Astala} for $\ds c\in L^{\infty}(\Omega)$ with measurements on the whole boundary in $\mathbb{R}^2$. To estimate the Robin coefficient in a stationary diffusion equation, Zou and Jin \cite{Zou2} developed a suitable finite-element method by considering boundary measurements of the solution and the heat flux. Using complex exponentially solutions, Calderon \cite{Calderon}, and Sylvester and Uhlmann \cite{SU} showed uniqueness for the diffusion coefficient in $\mathbb{R}^3$. Yamamoto \cite{Yamamoto1} realized Lipschitz stability results for parabolic equations. But by closely related approaches, Benabdallah, Gaitan and Le Rousseau \cite{Gaitan1} proved a Lipschitz stability result for the determination of a piecewise-constant diffusion coefficient.\\
The paper is organized as follows. In Section 2, we
formulate our main problem and we introduce the perturbed problem of heat conduction. Moreover, we describe briefly our inverse problem treated
in this connection. In Section 3, we prove rigourously some energy estimates, associated to the temperature distribution, which will be useful for our future results. Section 4 is devoted to
the reconstruction method in order to recover the small perturbations in the heat conductivity distribution which are quantified by the function $\emph{\textbf{c}}(x)$. The reconstruction theorem, completely proved in this section, is deeply based on appropriate averaging using particular background solutions constructed by an exact control method related to parabolic equations. Finally, in Section 5, we conclude our work and we suggest
that our methods can be useful, in a forthcoming investigation, to
identify diffusion coefficient in an anisotropic and/or in a non-cartesian medium.
\section{Presentation of the Inverse Problem} Let $\Omega\subset\R^d$, $d=2,3$ be a bounded, smooth domain with boundary $\partial \Omega$ of class ${\cal C}^{2}$.
By $\nu = \nu(x)$ we denote the outward unit normal vector to
$\Omega$ at a point $x\in\partial \Omega$, and we set $\displaystyle \partial_\nu u =\frac{\partial u}{\partial \nu}=\nabla u\cdot \nu$ and $\partial_t=
\ds\frac{\partial}{\partial t}$. Let $\Omega^\prime$ be a smooth subdomain of
$\Omega$ and is isotropic, \emph{i.e.} its thermal conductivity is the same in all directions. Let $\Gamma \subset \subset
\partial \Omega$ denote a measurable smooth connected part of the boundary
$\partial \Omega$. $\Gamma$ may be the \emph{accessible} part of $\partial\Omega$, on which we can make our measurements.\\

We suppose that $\Omega$ is occupied by a material of a positive thermic conductivity
\begin{equation}\label{c-perturbation}
\emph{\textbf{c}}_\alpha(x) = \emph{\textbf{c}}_0 + \alpha \emph{\textbf{c}}(x), x \in \Omega.
\end{equation}
The positivity of the body's thermal coefficient is required on both physical and
mathematical grounds. We assume that
\[\emph{\textbf{c}}(x) \in {\cal C}^1(\overline{\Omega})\cap{\cal C}^2(\Omega),
\emph{\textbf{c}} \equiv 0 {\rm \;in\;} \Omega \setminus
\overline{\Omega^\prime}.\] We also assume that $\alpha
>0$, the order of magnitude of the small perturbations, is sufficiently small
 that \be \label{ac} |\emph{\textbf{c}}_\alpha(x)|
\geq \kappa > 0, x \in \overline{\Omega},\quad \mbox{for }0<\alpha<\alpha_0 \ee where $\kappa$ and $\alpha_0$ are positive constants.\\

Let $u(x,t)$ be the solution of the initial boundary value problem for the heat equation in the absence of perturbations ($\alpha=0$):
\begin{equation}\label{bvp}
\left\{
  \begin{array}{ll}
    \partial_tu-\emph{\textbf{c}}_0\Delta u=0, & (x,t)\in \Omega\times[0,T] \\
  u(x,0)=\varphi(x), & x\in\Omega \\
   u(x,t)|_{\partial\Omega\times[0,T]}=f(x,t),
  \end{array}
\right.
\end{equation}
where the regular data $\varphi$ and $f$ are known.\\
Physically, we consider a heat-conducting body modeled by the set $\overline{\Omega}$
 and the
strictly positive heat conductivity distribution $\emph{\textbf{c}}_0$ inside the body. Here $T > 0$ is a given final time. The function
$\varphi$ is the initial temperature distribution in $\Omega$
 over which we do not have control.\\
Let $u_\alpha(x,t)$ denote the solution of the initial boundary value problem for the heat equation in the presence of the linear perturbations (\ref{c-perturbation}):
\begin{equation}\label{bvp-alpha}
\left\{
  \begin{array}{ll}
    \partial_tu_\alpha-(\nabla\cdot \emph{\textbf{c}}_\alpha \nabla)u_\alpha=0, & (x,t)\in \Omega\times[0,T] \\
  u_\alpha(x,0)=\varphi(x), & x\in\Omega \\
   u_\alpha(x,t)|_{\partial\Omega\times[0,T]}=f(x,t).\\
  \end{array}
\right.
\end{equation}
We perform boundary measurements by applying the temperature $f(x,t)$ at the
boundary $\partial\Omega$ during the time $0 < t < T$ and measuring the resulting heat flux $\displaystyle\frac{\partial u_\alpha}{\partial\nu}|_{\partial\Omega}$
 through the boundary; where $u_\alpha$ is the solution of (\ref{bvp-alpha}).\\

Throughout this paper, we shall use quite standard $L^2$-based Sobolev spaces to measure regularity. The notation $ H^s$ is used to denote
 those functions who along with all their derivatives of order less than and equal to s are in $L^2$.  $H^1_0$ denotes the closure of ${\cal C}^\infty_0$ in the norm
of $H^1$. Sobolev spaces with negative indices are in general defined by duality,
using $L^2$-inner product. 
It is classical to prove that the perturbed problem for the heat equation (\ref{bvp-alpha}) has a unique weak
solution $u_\alpha\in H^{2,1}(\Omega\times[0,T])$ (see for example, \cite{Hsiao,LM}), where the anisotropic Sobolev space $H^{2,1}(\Omega\times[0,T])$ is defined by
$$H^{2,1}(\Omega\times[0,T]):=L^2\big([0,T];H^2(\Omega)\big)\cap H^1\big([0,T];L^2(\Omega)\big).
$$
More details and comments about general anisotropic Sobolev space $H^{r,s}(\Omega\times[0,T])$ (for $r\geq0$, and $s\geq0$) can be found in the well-known works of Lions and Magenes in \cite{LM}.\\

In this article, we propose to solve the following inverse problem:\\
\textbf{Inverse Problem.} {\it Given a time $T > 0$, boundary data $f$ and initial data $\varphi$,
reconstruct the function $\emph{\textbf{c}}(x)$ for $x\in\Omega'$, defined by (\ref{c-perturbation})-(\ref{ac}), from only knowledge of boundary
measurements of $
\displaystyle \partial_\nu u_\alpha \quad {\rm on\;} \Gamma \times (0, T),$ i.e., on the (accessible) part $\Gamma$ of the boundary $\partial \Omega$ and
on the finite interval in time $(0, T)$ and where $ u_\alpha$ is the solution to problem (\ref{bvp-alpha}).}\\
For this purpose, we develop an asymptotic method based on appropriate averaging
using particular background solutions as weights. These particular solutions are
constructed by a control method as it has been done in the original work
\cite{Yamamoto1} concerning the inverse source hyperbolic problem, and also we can refer to \cite{Ammari,Ammari1} for the case of electromagnetic problem.
 By means of specific test functions our main result can be read as an approximation to the Fourier transformation which may be suggested as an idea for a numerical reconstruction algorithm.\\
The above inverse boundary value problem is related to nondestructive testing
where one looks for anomalous materials inside a known material. 
A similar approach may be applied to the "perturbed" full
(time-dependent) Stokes equations with nonconstant parameters.
 This may be discussed in a
forthcoming work.\\ 
The underlined inverse problem differs considerably from that considered by Ammari et al. in \cite{Ammari3} where the authors determined
an internal thermal conductivity of a given object. Moreover, our inverse problem differs from that evoked by Zou and Engl in the well-known work \cite{Zou1} where the authors presented a new approach, by using Tikhonov regularization, in order to identify the conductivity distribution in a heat conduction system.\\

There are lots of works on inverse problem of heat conductivity, see \cite{Jbalia1,Gaitan1,bryan,Gaitan2,ikehata1,ikehata2,jia} and the references therein.\\
Generally, the determination of conductivity profiles from knowledge
of boundary measurements has received a great deal of attention
(see for example, the important works elaborated by Ammari et al. \cite{Ammari2,Ammari3,Ammari4,Ammari5}, and those of Vogelius et al. in \cite{FV,VV}). The reconstruction of "perturbed" thermal conductivity within dynamics is much less investigated. For discussions
 on other interesting inverse hyperbolic problems, the reader is referred
for example to
Isakov \cite{I}, Puel and Yamamoto \cite{PY},
Bruckner and Yamamoto \cite{BY}.\\

\section{An Energy Estimate}
In this section we may estimate the difference, between the solution $u_\alpha$ of the perturbed heat equation and the exact background solution $u$, with respect to the order of magnitude of the small perturbations $\alpha$. To do it, we can define the function $v_\alpha\in C^1\big([0,T];L^2(\Omega)\big)\cap L^2\big(0,T;H_0^1(\Omega)\big) $ to be the solution of:
\begin{equation}\label{eq-v-alpha}
\left\{
  \begin{array}{ll}
    \nabla\cdot\ \emph{\textbf{c}}_\alpha\nabla v_\alpha=u_\alpha-u; & \hbox{in $\Omega\times[0,T]$,} \\
     v_\alpha=0; & \hbox{in $\partial\Omega\times[0,T]$,} \\
    v_\alpha(x,0)=0; & x\in\Omega.
  \end{array}
\right.
\end{equation}

The existence and uniqueness of $v_\alpha$ is most established by variational means.\\

The following energy estimate of $u_\alpha-u$ holds.
\begin{proposition}\label{prop1}Suppose that we have all hypothesis (\ref{c-perturbation})-(\ref{ac}). Assume that $\emph{\textbf{c}}_0$ is a constant function in $\Omega'$. Let $d=2$ and $\alpha_0$ be the constant given by (\ref{ac}). Then, there exist $C>0$ such that, for $0<\alpha<\alpha_0$ the
following energy estimate holds:$$\|\nabla(u_\alpha-u)\|_{L^\infty(0,T;L^{2}(\Omega))}+\|u_\alpha-u\|_{L^\infty(0,T;L^2(\Omega))}\leq C\alpha.$$
The constants
$C$ dependent on the domain $\Omega$, $\kappa$, $\alpha_0$, $T$, the
data $\varphi$ and $f$, but are otherwise independent of
$\alpha$.
\end{proposition}
\begin{proof}Let $v:(x,t)\in\Omega\times [0,T]\mapsto v(x,t)\in \mathbb{R}$, be any function such that $v\in H_0^1(\Omega)$ and simplified $v\equiv v(\cdot,t) : \Omega\to \mathbb{R }$. The main achievement of the proof is the following equality
\begin{equation}\label{energy1}\int_\Omega\partial_t(u_\alpha-u)v+\int_\Omega \emph{\textbf{\textbf{c}}}_\alpha\nabla(u_\alpha-u).\nabla v=\int_{\Omega}(\emph{\textbf{c}}_0-\emph{\textbf{c}}_{\alpha})\nabla u.\nabla v.\end{equation}
To do this, one can remark that $u_\alpha(\cdot,t)-u(\cdot,t)\in H_0^1(\Omega)$, and by applying Green's formula, we obtain$$\int_\Omega\partial_t(u_\alpha-u)v=\int_\Omega(\nabla\cdot(\emph{\textbf{c}}_\alpha\nabla u_\alpha)-\emph{\textbf{c}}_0\Delta u)v=
-\int_\Omega \emph{\textbf{c}}_\alpha\nabla u_\alpha.\nabla v
+\int_{\partial\Omega} \emph{\textbf{c}}_\alpha\partial_\nu u_\alpha v$$
$$+\int_\Omega \emph{\textbf{c}}_0\nabla u.\nabla v-\int_{\partial\Omega} \emph{\textbf{c}}_0\partial_\nu u v.$$
Note that $v\in H_0^1(\Omega)$, then we get$$\int_\Omega\partial_t(u_\alpha-u)v=-\int_\Omega \emph{\textbf{c}}_\alpha\nabla u_\alpha.\nabla v+\int_\Omega \emph{\textbf{c}}_0\nabla u.\nabla v.$$
Hence,\begin{equation}\label{integg1}
\int_\Omega\partial_t(u_\alpha-u)v=-\int_\Omega \emph{\textbf{c}}_\alpha\nabla(u_\alpha-u).\nabla v+\int_\Omega(\emph{\textbf{c}}_0-\emph{\textbf{c}}_\alpha)\nabla u.\nabla v,
\end{equation}
which is the desired relation (\ref{energy1}).\\
On the other hand, let $v_\alpha$ be the solution of (\ref{eq-v-alpha}). Then, we can write
$$\int_\Omega \emph{\textbf{c}}_\alpha\nabla(u_\alpha-u).\nabla v_\alpha=-\int_\Omega (\nabla\cdot\ \emph{\textbf{c}}_\alpha\nabla v_\alpha)(u_\alpha-u)=-\int_\Omega |u_\alpha-u|^2\leq0.$$
Next,
$$
\int_\Omega\partial_t(u_\alpha-u)v_\alpha=\int_\Omega \partial_t(\nabla\cdot\ \emph{\textbf{c}}_\alpha\nabla v_\alpha)v_\alpha=-\int_\Omega \emph{\textbf{c}}_\alpha\partial_t\nabla v_\alpha.\nabla v_\alpha=-\frac{1}{2}\partial_t\int_\Omega \emph{\textbf{c}}_\alpha|\nabla v_\alpha|^2.
$$
Adding both sides, we obtain from (\ref{energy1}) with $v$ replaced by $v_\alpha$:
$$\partial_t\int_\Omega \emph{\textbf{c}}_\alpha|\nabla v_\alpha|^2+2\int_\Omega |u_\alpha-u|^2=-2\int_{\Omega}(\emph{\textbf{c}}_0-\emph{\textbf{c}}_\alpha)\nabla u.\nabla v_\alpha.$$
 This implies, by considering (\ref{c-perturbation}), that
 \be\label{eqqq}\partial_t\int_\Omega \emph{\textbf{c}}_\alpha|\nabla v_\alpha|^2\leq2|\int_{\Omega}(\emph{\textbf{c}}_0-\emph{\textbf{c}}_\alpha)\nabla u.\nabla v_\alpha|\leq C\alpha\|\emph{\textbf{c}}\|_{L^{\infty}(\Omega')}.\|\nabla u(.,t)\|_{L^2(\Omega)}\|\nabla v_\alpha(.,t)\|_{L^2(\Omega)}.\ee
 On the other hand, using (\ref{c-perturbation})-(\ref{ac}), one can find constants $M_0>0$, and $M_1>0$ such that
 \be\label{eqqq-1}\displaystyle\sup_{x\in \Omega}|\emph{\textbf{c}}(x)|= M_0,\quad \mbox{and }\kappa\leq \sup_{x\in \Omega}|\emph{\textbf{c}}_\alpha(x)|\leq M_1:=\emph{\emph{\textbf{c}}}_0+\alpha_0 M_0,\quad \mbox{for }0<\alpha<\alpha_0.\ee
Then, (\ref{eqqq}) becomes
$$\partial_t\int_\Omega \emph{\textbf{c}}_\alpha|\nabla v_\alpha|^2\leq C \alpha\|\nabla v_\alpha(.,t)\|_{L^2(\Omega)}.
$$
According to Gronwall lemma, we get
\be\label{eqqq-2}\|\nabla v_\alpha(.,t)\|_{L^2(\Omega)}\leq C\alpha.
\ee
On the other hand, replacing $v$ by $u_\alpha-u$ in equality $(\ref{energy1})$ and taking into account (\ref{eq-v-alpha}) and (\ref{eqqq-1})-(\ref{eqqq-2}), we immediately obtain $$\frac{1}{2}\partial_t\int_\Omega|u_\alpha-u|^2+\int_\Omega \emph{\textbf{c}}_\alpha|\nabla(u_\alpha-u)|^2=\int_{\Omega}(\emph{\textbf{c}}_0-\emph{\textbf{c}}_{\alpha})\nabla u.\nabla(u_\alpha-u)$$
$$=-\alpha\int_{\Omega}\emph{\textbf{c}}\nabla u.\nabla(u_\alpha-u)$$
$$=\alpha\int_{\Omega}[\emph{\textbf{c}}\Delta u+\nabla(\emph{\textbf{c}})\cdot\nabla u](\nabla. \ \emph{\textbf{c}}_\alpha\nabla v_\alpha)$$
$$\leq C\alpha\|\nabla v_\alpha(.,t)\|_{L^2(\Omega)}\leq C\alpha^{2}.$$
It is easily seen that,
$$
 \displaystyle\partial_t\int_\Omega|u_\alpha-u|^2\leq C\alpha^{2},\quad \int_\Omega|\nabla(u_\alpha-u)|^2\leq C\alpha^{2}
  $$
and consequently,
$$ \displaystyle\|u_\alpha-u\|_{L^\infty(0,T;L^2(\Omega))}\leq C_1\alpha,\quad
                    \displaystyle\|\nabla(u_\alpha-u)\|_{L^\infty(0,T;L^2(\Omega))}\leq C_2\alpha. $$

Thus, we obtain:$$\|u_\alpha-u\|_{L^\infty(0,T;L^2(\Omega))}+\|\nabla(u_\alpha-u)\|_{L^\infty(0,T;L^2(\Omega))}\leq C\alpha,$$where $C$ is independent of $\alpha$.\square   \end{proof}

\section{The Reconstruction Method}
To recover the small perturbations in the heat conductivity, which are quantified by the function $\emph{\textbf{c}}$, let us introduce the
 following cutoff function $\beta(x) \in
{\cal C}^{\infty}_0(\Omega)$ such that
$\beta \equiv 1$ in a subdomain
$\Omega^{\prime}$ of $\Omega$ that contains the perturbations. For an arbitrary $\eta \in\mathbb{ R}^2$, we assume
that we are in possession of the boundary measurements: \begin{equation}\label{mesur-isentification}\displaystyle
\partial_\nu u \mbox{ and }\partial_\nu u_\alpha\quad {\rm on\;} \Gamma
\times (0, T)\end{equation}for the data $$ \displaystyle\varphi(x) = \varphi_\eta (x) = \frac{J_0(\frac{z}{a|\eta|}|x|)}{z J_1(z)},\quad \mbox{and }f(x, t) = f_\eta(x, t) = \frac{J_0(\frac{z}{a|\eta|}|x|)}{z J_1(z)}e^{-(\frac{z}{a|\eta|})^2\emph{\textbf{c}}_0t}.$$
Here $z$ is the positive zero of the Bessel function of the first kind $J_0(x)$ (\cite{Bowman}, pp. 37-39), $J_1(x)$ is a Bessel function, and $a:=1/2\max\{dist(x,y):x,y\in\Omega\}$.\\
This particular choice of data $\varphi$ and $f$ allows us to give explicitly the background solution $u$ of the heat equation (\ref{bvp}) in the
absence of any perturbations. 
 To solve the underlined inverse problem, it suffices to record the boundary measurements of $\partial_\nu u_\alpha $, because the measure of $\partial_\nu u $ deducts according to classic results.\\

The reconstruction method, based on the knowledge of the boundary measurements (\ref{mesur-isentification}),
requires the resolution of an exact boundary control problem for the heat equation.\\
Unlike the analogue problem for the wave equation in \cite{Lions,Glowinski1, Glowinski2} or for the Maxwell problem \cite{Ammari, Darbas1}, this problem is ill-posed in general. To
overcome the ill-posedness, Carthal, Glowinski and Lions used in their original work \cite{Glowinski1} two regularization procedures for which the corresponding control problems are well posed.\\
We suppose here that the boundary control is of the Dirichlet type but the Neumann's case, may be handled by a similar manner.\\
Indeed, for $\eta \in \mathbb{R}^2$, the controllability problem consists in finding a scalar function $g_\eta$ such that $\phi_{\eta}(T)  = 0$ in $\Omega$, where $\phi_{\eta}$ is the unique weak solution of the problem
 \begin{eqnarray}
\displaystyle\nonumber
(\partial_t - \emph{\textbf{c}}_0 \Delta) \phi_{\eta}  = 0 \quad {\rm in}\;
\Omega \times (0, T),\\\label{beta-eta}
\phi_{\eta}  |_{t=0} = \beta(x) e^{i \eta \cdot x} \in H^1_0(\Omega), \\\label{weta}
\phi_{\eta} |_{\Gamma  \times (0, T)} = g_\eta,\\
\nonumber
\phi_{\eta} |_{\partial \Omega \setminus \overline{\Gamma}
  \times (0, T)} = 0.
  \end{eqnarray}
  In the spirit of the Hilbert Uniqueness Method, Glowinski et al. \cite{Glowinski1,Glowinski2} proved the existence of the control $g_\eta\in L^2(0,T;H^{-1/2}(\Gamma))$ in such a way that $\phi_{\eta}\in L^2(0,T;L^2(\Omega))$ and $\ds \frac{\partial\phi_\eta}{\partial t}\in L^2(0,T;H^{-2}(\Omega))$ (see e.g., \cite{Glowinski1} page 124).\\
The final condition $\phi_\eta(T)=0$ is simply deduced by making change of variable $y(T-t)$ for the final state $y(T)$.
 \subsection{Main Results}
 In this paragraph, we will determine the procedure to identify the heat conductivity in $\Omega$. Let $\eta \in \mathbb{R}^2$ and considering the function $v_{\eta} \in H^{2,1}(\Omega\times[0,T])$ satisfying the following state equation:\begin{eqnarray}
\displaystyle\nonumber
    \partial_tv_{\eta}+\emph{\textbf{c}}_0\Delta v_{\eta}=0,\quad (x,t)\in \Omega\times[0,T] & \\
    \label{v-alpha-0}
  v_{\eta}(x,0)=i \nabla\cdot\big(\eta \emph{\emph{\textbf{c}}}(x)e^{i\eta.x}\big) \in L^2(\Omega),\quad  x\in\Omega &\\\nonumber
   v_{\eta}|_{\partial\Omega\times[0,T]}=0,\quad \mbox{and }\partial_t v_{\eta}(x,0)=0,\quad  x\in\partial\Omega.& \\\nonumber
  \end{eqnarray}
  The existence and uniqueness of the solution
$v_{\eta} $ can be established by transposition,
see (\cite{Hsiao}, pp. 106-107) or \cite{Evans,LM}.\\
To determine our procedure, we need the following proposition:
\begin{proposition}\label{prop3-1}Let $\Gamma\subset\subset\partial\Omega$ be an accessible part. Suppose that we have all hypothesis (\ref{c-perturbation})-(\ref{ac}), and assume that $\emph{\textbf{c}}_0$ is a constant function in $\Omega'$. For any $\eta\in\mathbb{R}^2$, we have the following result:
$$\int_0^T\int_{\Gamma}\partial_\nu v_{\eta}g_\eta~ds(x)dt=\frac{|\eta|^2}{\emph{\textbf{c}}_0}\int_{\Omega^{\prime}}\emph{\textbf{c}}(x)e^{2i\eta.x}dx,$$
where $g_\eta$ is given by (\ref{weta}) and $v_\eta$ is the solution of (\ref{v-alpha-0}).
\end{proposition}
\begin{proof}
Since $\partial_tv_{\eta}+\emph{\textbf{c}}_0\Delta v_{\eta}\equiv0$ in $\Omega \times (0, T)$, then
$$\int_0^T\int_\Omega\partial_tv_{\eta}\phi_{\eta}+\emph{\textbf{c}}_0\int_0^T\int_\Omega\Delta v_{\eta}\phi_{\eta}=0.$$

Using Green's formula, we obtain:$$\int_\Omega\Delta v_{\eta}\phi_{\eta}=-\int_\Omega\nabla v_{\eta}\nabla \phi_{\eta}+\int_{\Gamma}g_\eta\partial_\eta v_{\eta}.$$
Therefore, $$\int_0^T\int_\Omega\partial_t v_{\eta}\phi_{\eta}-\emph{\textbf{c}}_0\int_0^T\int_\Omega\nabla v_{\eta}\nabla \phi_{\eta}=-\emph{\textbf{c}}_0\int_0^T\int_{\Gamma}g_\eta\partial_\eta v_{\eta}.$$
Integrating by parts, we obtain
\begin{equation}\label{eqq1}
-\int_\Omega v_{\eta}(x,0)\phi_{\eta}(x,0)dx-\int_0^T\int_\Omega v_{\eta}\partial_t \phi_{\eta}-\emph{\textbf{c}}_0\int_0^T\int_\Omega\nabla v_{\eta}\cdot\nabla \phi_{\eta}
\end{equation}
$$=-\emph{\textbf{c}}_0\int_0^T\int_{\Gamma}g_\eta\partial_\nu v_{\eta}.
$$
On the other hand,$$-\int_0^T\int_\Omega v_{\eta}\partial_t \phi_{\eta}=-\int_0^T\int_\Omega v_{\eta}\emph{\textbf{c}}_0\Delta \phi_{\eta}=-\emph{\textbf{c}}_0\int_0^T\int_{\partial\Omega}v_{\eta}\partial_\nu \phi_{\eta}+\emph{\textbf{c}}_0\int_0^T\int_\Omega\nabla v_{\eta}\cdot\nabla \phi_{\eta}
       $$
       $$=\emph{\textbf{c}}_0\int_0^T\int_\Omega\nabla v_{\eta}\cdot\nabla \phi_{\eta}.$$
Then, relation (\ref{eqq1}) becomes $$-\int_\Omega v_{\eta}(x,0)\phi_{\eta}(x,0)dx=-\emph{\textbf{c}}_0\int_0^T\int_{\Gamma}\partial_\eta v_{\eta}g_\eta,$$
Then, taking (\ref{v-alpha-0}) and (\ref{beta-eta}) into account we immediately obtain
$$\displaystyle \int_{\Omega}i \nabla\cdot(\eta \emph{\emph{\textbf{c}}}(x)e^{i\eta.x})\beta(x) e^{i\eta.x}dx=\emph{\textbf{c}}_0\int_0^T\int_{\Gamma}\partial_\eta v_{\eta}g_\eta~ds(x)dt.$$
Thus, by integrating by parts
$$|\eta|^2\int_{\Omega^{\prime}}\emph{\textbf{c}}(x)e^{2i\eta.x}dx=\emph{\textbf{c}}_0\int_0^T\int_{\Gamma}\partial_\eta v_{\eta}g_\eta~ds(x)dt$$
which completes the proof.\square\\
\end{proof}
From the previous definitions, we can define the following,
\begin{equation}\label{u-tilde}
\ds \tilde{u}_{\alpha}(x, t) = u(x, t) + \alpha \int_0^t e^{- i\sqrt{\emph{\textbf{c}}_0} |
\eta | s} v_{\eta}(x, t-s)\; ds,\quad x \in \Omega, t \in (0, T).
\end{equation}
The following estimation holds.
\begin{theorem}\label{thm1}Assume that $\emph{\textbf{c}}_\alpha$ is defined by (\ref{c-perturbation}). Let $u$ and $v_\eta$ be the solutions of (\ref{bvp}) and (\ref{v-alpha-0}) respectively. Then the function $\tilde{u}_{\alpha}$ given by (\ref{u-tilde}) is well defined, and there exist some constants $\alpha_1>0$ and $C>0$ such that for $0<\alpha<\alpha_1$ we have:\be
\label{r6} \ds || \partial_n u_\alpha -\partial_n
\tilde{u}_{\alpha}) ||_{L^2(0, T; L^2(\Gamma))} \leq C\alpha^2.\ee
Here $u_{\alpha}$ is the solution of (\ref{bvp-alpha}), $C$ dependent on $\Omega,\Gamma$ and $T$ but independent of $\alpha$.
\end{theorem}
To prove Theorem \ref{thm1}, we should use the following lemma.
\begin{lemma}\label{lem1}Assume that we have all hypothesis of Theorem \ref{thm1}. Let the function $\tilde{u}_{\alpha}$ given by (\ref{u-tilde}). Then, the following nonhomogeneous heat equations are well defined:
\begin{itemize}
  \item [1)] For $(x,t)\in\Omega \times (0, T)$, $\tilde{u}_{\alpha}$ satisfies: $$(\partial_t - \emph{\textbf{c}}_0\Delta) \tilde{u}_\alpha  =
i \alpha \nabla\cdot (\eta \emph{\textbf{c}}(x) e^{i \eta \cdot x}),\quad\tilde{u}_\alpha(x,0) = e^{i \eta \cdot x}\mbox{ and }\tilde{u}_\alpha  |_{\partial \Omega  \times (0, T)}
= e^{i \eta \cdot x - i |\eta| t}.$$
  \item [2)]For $(x,t)\in\Omega \times (0, T)$, $u_{\alpha}-\tilde{u}_{\alpha}$ satisfies:
  \begin{equation}\label{r2}
\ds \big(\partial_t - \nabla\cdot \emph{\textbf{c}}_\alpha \nabla \big) (u_\alpha -
\tilde{u}_\alpha) = \alpha^2 \nabla\cdot \big(\emph{\textbf{c}}(x) \nabla (\int_0^t e^{-
i | \eta | s} v_{\eta}(x, t-s)\; ds) \big),\end{equation}
$(u_\alpha - \tilde{u}_\alpha)(x,0)= 0,$ and $(u_\alpha - \tilde{u}_\alpha)|_{\partial \Omega  \times (0, T)}
= 0.$
\end{itemize}
\end{lemma}
\begin{proof}Let $v_\eta$ be the solution of (\ref{v-alpha-0}). By using a variable change, we get $$\partial_t\int_0^te^{-i\sqrt{\emph{\textbf{c}}_0}|\eta|s}v_{\eta}(x,t-s)ds=-i\sqrt{\emph{\textbf{c}}_0}|\eta|e^{-i\sqrt{c_0}|\eta|t}\int_0^t{e^{i\sqrt{c_0}|\eta|s'}v_{\eta}(x,s')ds'}
+v_{\eta}(x,t),$$where $(x,t)\in\Omega\times(0,T)$. Therefore,\be\label{equ1}(\partial_t-\emph{\textbf{c}}_0\Delta)(\int_0^te^{-i\sqrt{\emph{\textbf{c}}_0}|\eta|s}v_{\eta}(x,t-s)ds=
-i\sqrt{\emph{\textbf{c}}_0}|\eta|e^{-i\sqrt{\emph{\textbf{c}}_0}|\eta|t}\int_0^te^{i\sqrt{\emph{\textbf{c}}_0}|\eta|s}v_{\eta}(x,s)ds+
v_{\eta}(x,t)
\ee
$$+\emph{\textbf{c}}_0e^{-i\sqrt{\emph{\textbf{c}}_0}|\eta|t}\int_0^te^{i\sqrt{\emph{\textbf{c}}_0}|\eta|s}\Delta{v_{\eta}(x,s)}ds.$$
Now, integrating by parts:
$$\int_0^te^{i\sqrt{\emph{\textbf{c}}_0}|\eta|s}v_{\eta}(x,s)ds=\frac{e^{i\sqrt{\emph{\textbf{c}}_0}|\eta|t}}{i\sqrt{\emph{\textbf{c}}_0}|\eta|}v_{\eta}(x,t)
-\frac{1}{i\sqrt{\emph{\textbf{c}}_0}|\eta|}
v_{\eta}(x,0)-\frac{1}{i\sqrt{\emph{\textbf{c}}_0}|\eta|}
\int_0^te^{i\sqrt{\emph{\textbf{c}}_0}|\eta|s}\partial_t v_{\eta}(x,s)ds.
$$
Inserting this last relation into (\ref{equ1}) and recall that $v_\eta$ is the solution of (\ref{v-alpha-0}). Then, we immediately obtain
$$(\partial_t-\emph{\textbf{c}}_0\Delta)(\int_0^te^{-i\sqrt{\emph{\textbf{c}}_0}|\eta|s}v_{\eta}(x,t-s)ds=
e^{-i\sqrt{\emph{\textbf{c}}_0}|\eta|t}\int_0^te^{i\sqrt{\emph{\textbf{c}}_0}|\eta|s}(\partial_t+\emph{\textbf{c}}_0\Delta)v_{\eta}(x,s)ds$$
$$+e^{-i\sqrt{\emph{\textbf{c}}_0}|\eta|t}v_{\eta}(x,0)=e^{-i\sqrt{\emph{\textbf{c}}_0}|\eta|t}v_{\eta}(x,0).
$$
Hence, to achieve the proof of {\it 1)}, one may use (\ref{u-tilde}) and recall that $u$ solves (\ref{bvp}).\\
To prove {\it 2)},  it suffices to handle the first equation in (\ref{r2}), because the other relations can be deduced easily.
Recall that $u_{\alpha}$ solves the problem (\ref{bvp-alpha}). Then, we get
$$
\big(\partial_t - \nabla\cdot \emph{\textbf{c}}_\alpha \nabla \big) (u_\alpha -
\tilde{u}_\alpha) =-\big(\partial_t - \nabla\cdot \emph{\textbf{c}}_\alpha \nabla \big)\tilde{u}_\alpha=-\big(\partial_t - \nabla\cdot (\emph{\textbf{c}}_0+\alpha \emph{\textbf{c}}) \nabla \big)\tilde{u}_\alpha=-\big(\partial_t - \alpha\nabla\cdot (\emph{\textbf{c}}) \nabla \big)\tilde{u}_\alpha-i \alpha \nabla\cdot (\eta \emph{\textbf{c}}(x) e^{i \eta \cdot x})$$by using {\it 1)}.
Using (\ref{u-tilde}) and (\ref{bvp}), we can write
$$\big(\partial_t - \nabla\cdot \emph{\textbf{c}}_\alpha \nabla \big) (u_\alpha -
\tilde{u}_\alpha) =-\big(\partial_t - \alpha\nabla\cdot (\emph{\textbf{c}}) \nabla \big)u-\alpha\partial_t\big( \int_0^t e^{- i\sqrt{\emph{\textbf{c}}_0} |
\eta | s} v_{\eta}(x, t-s)\; ds     \big)$$
$$+\alpha^2(\nabla\cdot (\emph{\textbf{c}}) \nabla )\big(   \int_0^t e^{- i\sqrt{\emph{\textbf{c}}_0} |
\eta | s} v_{\eta}(x, t-s)\; ds    \big)       -i \alpha \nabla\cdot (\eta \emph{\emph{\textbf{c}}}(x) e^{i \eta \cdot x}),\quad x \in \Omega, t \in (0, T).$$
Thus, by using the relation (\ref{u-tilde}) we get
$$\big(\partial_t - \nabla\cdot \emph{\textbf{c}}_\alpha \nabla \big) (u_\alpha -
\tilde{u}_\alpha) =\alpha^2(\nabla\cdot \emph{\textbf{c}} \nabla )\big(   \int_0^t e^{- i\sqrt{\emph{\textbf{c}}_0} |
\eta | s} v_{\eta}(x, t-s)\; ds    \big),\quad x \in \Omega, t \in (0, T).$$
\square\\
Now, we are ready to prove Theorem \ref{thm1} by means of previous results.\\
{\bf Proof of Theorem \ref{thm1}:}\\
Let  $h\in{\cal C}^\infty_0(]0, T[)$ be an arbitrary function. For any $v \in L^1(0, T; L^2(\Omega))$ we define
\[
\ds \hat{v}(x) = \int_0^T v(x, t) h(t) \; dt.
\]
Hence, $\hat{v}\in  L^2(\Omega)$ and from (\ref{r2}) it follows that
$$\displaystyle(\nabla\cdot \emph{\textbf{c}}_\alpha\nabla)(\hat{u}_\alpha-\hat{\tilde{u}}_{\alpha})
=(\nabla\cdot\,\emph{\textbf{c}}_\alpha\nabla)\int_0^T(u_\alpha-\tilde{u}_{\alpha})h(t)dt
=\int_0^T\big(\partial_t(u_\alpha-\tilde{u}_{\alpha})-\alpha\nabla\cdot (\emph{\textbf{c}}(x) \nabla (\tilde{u}_\alpha - u))\big)h(t)dt.$$
Therefore,
\begin{equation}\label{equ2}\displaystyle(\nabla\cdot \emph{\textbf{c}}_\alpha\nabla)(\hat{u}_\alpha-\hat{\tilde{u}}_{\alpha})=
 -\alpha\nabla\cdot (\emph{\textbf{c}}(x) \nabla (\hat{u}_\alpha - \hat{u})) -
\int_0^T (u_\alpha - \tilde{u}_\alpha) h^{\prime }(t)\; dt \quad
{\rm in}\; \Omega,\mbox{ and }(\hat{u}_\alpha - \hat{\tilde{u}}_\alpha)   |_{\partial \Omega}
= 0.\end{equation}
Now, we try to estimate both quantities $||u_\alpha - \tilde{u}_\alpha
||_{L^\infty(0, T; L^2(\Omega))}$ and $|| \nabla(\hat{u}_\alpha -
\hat{u})||_{L^2(\Omega)}$. To arrive at this result, let $y_\alpha\in H^1(\Omega)$ be the solution of \be\label{y-alpha}\nabla\cdot
\emph{\textbf{\textbf{c}}}_\alpha \nabla y_\alpha = \partial_t (u_\alpha - u) \quad
{\rm in\;} \Omega\quad \mbox{and }y_\alpha =0\quad \mbox{on }\partial\Omega.
\ee
As done in the proof of Proposition \ref{prop1}, Green's formula yields:

\[\ds \int_{\Omega} \partial_t (u_\alpha -u) y_\alpha
 + \int_{\Omega}
\emph{\textbf{c}}_\alpha \nabla(u_\alpha -u) \cdot \nabla y_\alpha  = -\alpha
\int_\Omega \emph{\textbf{c}}  \nabla u \cdot \nabla y_\alpha.
\]
While
\[
\ds  \int_{\Omega} \emph{\textbf{c}}_\alpha \nabla (u_\alpha -u) \cdot \nabla
y_\alpha = -\int_{\Omega} (u_\alpha -u)\nabla\cdot(\emph{\textbf{\textbf{c}}}_\alpha \nabla y_\alpha) = - \int_{\Omega} (u_\alpha -u)\partial_t (u_\alpha - u) =
- \frac{1}{2} \partial_t \int_{\Omega} (u_\alpha - u)^2,
\]
as well as
\[
\ds  \int_{\Omega} \partial_t (u_\alpha -u) y_\alpha =\int_{\Omega} \nabla\cdot(\emph{\textbf{\textbf{c}}}_\alpha \nabla y_\alpha) y_\alpha =
 - \int_{\Omega}
\emph{\textbf{c}}_\alpha |\nabla y_\alpha|^2,
\]
we obtain
\[
\ds
2 \int_{\Omega}
\emph{\textbf{c}}_\alpha |\nabla y_\alpha|^2 +
\partial_t \int_{\Omega} (u_\alpha - u)^2 = 2 \alpha
\int_\Omega \emph{\textbf{c}}  \nabla u \cdot \nabla y_\alpha \leq C \alpha
||\nabla y_\alpha||_{L^\infty(0, T; L^2(\Omega))}.
\]
From the Gronwall Lemma it follows that
\be
\label{u1} || u_\alpha - u ||_{L^\infty(0, T; L^2(\Omega))}  \leq
C \alpha. \ee
As a result, the function $\hat{u}_\alpha - \hat{u}$ satisfies
\[\nabla\cdot\big( \emph{\textbf{c}}_\alpha \nabla (\hat{u}_\alpha - \hat{u})\big) = O(\alpha)\quad {\rm in\;} \Omega, \mbox{ and } \hat{u}_\alpha - \hat{u}=0\quad {\rm on \;} \partial\Omega,\]
and so, by considering (\ref{eqqq-1}) we obtain, for $0<\alpha<\alpha_0$, that
\be
\label{r40} || \nabla(\hat{u}_\alpha - \hat{u})||_{L^2(\Omega)} =
O(\alpha). \ee
 However $\nabla (u_\alpha - u) \in L^{\infty}(0, T;
L^2(\Omega))$ which gives by using the above estimate that
\be
\label{r420} || \nabla(u_\alpha - u)||_{L^2(\Omega)} = O(\alpha)\quad
{\rm \; a.e. \; } t \in (0, T). \ee
 As defined in (\ref{y-alpha}), let us introduce the following
function $\tilde{y}_\alpha\in H^1(\Omega)$ to verify:
\[\nabla\cdot(
\emph{\textbf{\textbf{c}}}_0 \nabla \tilde{y}_\alpha) = \partial_t (\tilde{u}_\alpha - u_\alpha) \quad
{\rm in\;} \Omega\quad \mbox{and }\tilde{y}_\alpha =0\quad \mbox{on }\partial\Omega.\]

By means of (\ref{r2}), we compute that
\[
\ds\int_{\Omega}
\emph{\textbf{\textbf{c}}}_\alpha |\nabla \tilde{y}_\alpha|^2 +
\partial_t \int_{\Omega} (\tilde{u}_\alpha - u_\alpha)^2 = -2 \alpha
\int_\Omega \emph{\emph{\textbf{c}}}  \nabla (u - u_\alpha) \cdot \nabla
\tilde{y}_\alpha
\]
which, by using (\ref{r420}), yields
\be
\label{r50} ||  \tilde{u}_\alpha - u_\alpha
 ||_{L^\infty(0, T; L^2(\Omega))}  \leq C \alpha^2.
\ee Combining estimates (\ref{r40}) and (\ref{r50}) and using
standard elliptic regularity \cite{Evans} for the boundary value
problem (\ref{equ2}) we obtain \[\ds || \partial_\nu
 (\hat{u}_\alpha - \hat{\tilde{u}}_{\alpha}) ||_{L^2(\Gamma)} =
O(\alpha^2), \] and so, as for estimate (\ref{r40}), this permits
us to assert that
\[\ds || \partial_\nu u_\alpha -
\partial_\nu \tilde{u}_\alpha||_{L^2(0, T; L^2(\Gamma))} = O(\alpha^2). \]The theorem is proven
\square
\end{proof}\\
To identify the small perturbation of the heat conductivity $\emph{\textbf{c}}_\alpha$ let us view the averaging of the
boundary measurements of resulting heat flux $\displaystyle\partial_\nu u_\alpha
|_{\Gamma \times (0, T)}$, using the solution $\theta_\eta$ to the
Volterra equation of second kind, as a function of $\eta$:
\begin{equation}\label{eq4}
\left\{
\begin{array}{l}
\displaystyle\theta_\eta \in H^1(0, T; TL^2(\Gamma)),\\
\partial_t \theta_\eta (x, t) + \int_t^T e^{- i
\sqrt{\emph{\textbf{c}}_0} | \eta | (s -t)} ( \theta_\eta (x, s) - i
\sqrt{\emph{\textbf{c}}_0} | \eta | \partial_t \theta_\eta (x, s)) \; ds =
g_\eta(x, t);x \in \Gamma, t \in (0, T),\\
\theta_\eta(x, 0) = 0, \quad {\rm for\;} x \in \Gamma.
\end{array}
\right.
\end{equation}
The existence and uniqueness of $\theta_\eta$ in $H^1(0, T;
TL^2(\Gamma))$ for any $\eta \in \R^2$ can be established using
the resolvent kernel \cite{Yamamoto1}.\\
Now, by taking $t=T$ in last relation (\ref{eq4}), one can remark the following:
\begin{remark}\label{rem1}We have $\partial_t \theta_\eta (x, T)=g_\eta(x, T)$ for all $x \in \Gamma.$ While $g_\eta \in H^1_0(0, T; L^2(\Gamma))$, we get
$$\partial_t \theta_\eta (x, T)=0;\quad x \in \Gamma.$$
\end{remark}

The following main theorem permits us to reconstruct the
function
\[ \emph{\textbf{c}}(x) \in \{ \emph{\textbf{c}}  \in {\cal C}^1(\overline{\Omega}), \emph{\textbf{c}}
\equiv 0 {\rm \;on \;} \Omega \setminus \overline{\Omega^\prime},
|\emph{\textbf{c}}(x) | \leq \lambda, x\in \Omega^\prime  \},\] from the
boundary measurements of $\displaystyle \partial_\nu u_\alpha|_{\Gamma \times (0, T)}:=  \frac{\partial u_\alpha}{\partial \nu}
|_{\Gamma \times (0, T)}$.\\

Now, we can prove the following main result in this paper.
\begin{theorem}\label{main-thm}Suppose that we have all hypothesis (\ref{c-perturbation})-(\ref{ac}), and let $\theta_\eta$ be the solution to (\ref{eq4}). Let $u$, $u_\alpha$ be the unique solutions of the heat equations (\ref{bvp}) and (\ref{bvp-alpha}) respectively. If the heat conductivity $\emph{\textbf{c}}_0$ is constant in $\Omega'$, then for any $\eta\in\mathbb{R}^2$ we have:
\begin{equation}\label{mainresult}
\int_0^T\int_\Gamma[\theta_\eta(\partial_\nu u_\alpha-\partial_\nu u)+
\partial_t\theta_\eta\partial_t(\partial_\nu u_\alpha-\partial_\nu u)]
   = \alpha\frac{|\eta|^2}{\emph{\textbf{c}}_0}\int_{\Omega^{\prime}}\emph{\textbf{c}}(x)e^{2i\eta\cdot x}dx+O(\alpha^2).
   \end{equation}
   The term $O(\alpha^2)$ is independent of the function $\emph{\textbf{c}}$, but depends only on the bound $\lambda$.
  \end{theorem}

    \begin{proof}Let the function $\widetilde{u}_{\alpha}(x,t)$ defined by (\ref{u-tilde}). Then, by inserting $\widetilde{u}_{\alpha}$ into the left hand side of (\ref{mainresult}), we immediately get: $$\int_0^T\int_\Gamma[\theta_\eta(\partial_\nu u_\alpha-\partial_\nu u)+
    \partial_t\theta_\eta\partial_t(\partial_\nu u_\alpha-\partial_\nu u)]=
    \int_0^T\int_\Gamma[\theta_\eta(\partial_\nu u_\alpha-\partial_\nu\widetilde{u}_{\alpha})+
    \partial_t\theta_\eta\partial_t(\partial_\nu u_\alpha-\partial_\nu \widetilde{u}_{\alpha})]
$$ $$+\int_0^T\int_\Gamma[\theta_\eta\int_0^te^{i\sqrt{\emph{\textbf{c}}_0}|\eta|s}\partial_\nu v_{\eta}(x,t-s)ds+
\partial_t\theta_\eta\partial_t\int_0^te^{-i\sqrt{\emph{\textbf{c}}_0}|\eta|s}\partial_\nu v_{\eta}(x,t-s)ds].$$

On the other hand, we have $$\partial_t\theta(x,t)+\int_t^Te^{-i\sqrt{\emph{\textbf{c}}_0}|\eta|(s-t)}(\theta_\eta(x,s)-i\sqrt{\emph{\textbf{c}}_0}|\eta|\partial_t\theta_\eta(x,s))ds
=g_\eta(x,t).$$ By using a variable change and Remark \ref{rem1}, we get $$\partial_t\int_0^te^{-i\sqrt{\emph{\textbf{c}}_0}|\eta|s}\partial_\nu v_{\eta}(x,t-s)ds=
\partial_\nu v_{\eta}(x,t)-i\sqrt{c_0}|\eta|e^{-i\sqrt{\emph{\textbf{c}}_0}|\eta|t}\int_0^te^{i\sqrt{\emph{\textbf{c}}_0}|\eta|s}\partial_\nu v_{\eta}(x,z)dz.$$
Hence,
$$ \int_0^T\int_\Gamma[\theta_\eta\int_0^te^{i\sqrt{\emph{\textbf{c}}_0}|\eta|s}\partial_\nu v_{\eta}(x,t-s)ds+\partial_t\theta_\eta\partial_t\int_0^te^{-i\sqrt{\emph{\textbf{c}}_0}|\eta|s}\partial_\nu v_{\eta}(x,t-s)ds]
$$
$$
=\int_0^T\int_\Gamma[\theta_\eta\int_0^te^{i\sqrt{\emph{\textbf{c}}_0}|\eta|s}\partial_\nu v_{\eta}(x,t-s)ds+\partial_t\theta_\eta(-i\sqrt{c_0}|\eta|e^{-i\sqrt{\emph{\textbf{c}}_0}|\eta|t}\int_0^te^{i\sqrt{\emph{\textbf{c}}_0}|\eta|s}
\partial_\nu v_{\eta}(x,z)dz$$

$$+\partial_\nu v_{\eta}(x,t))]=\int_0^T\int_\Gamma\partial_\nu v_{\eta}(x,t)[\partial_t\theta_\eta+\int_t^T(\theta_\eta(z)-i\sqrt{\emph{\textbf{c}}_0}|\eta|\partial_t\theta_\eta(z))e^{i\sqrt{\emph{\textbf{c}}_0}|\eta|(t-z)}dz]dt
$$
$$=\int_0^T\int_\Gamma{g_\eta(x,t)\partial_\nu v_{\eta}(x,t)dt}.$$
 Consequently,
 $$\int_0^T\int_\Gamma[\theta_\eta(\partial_\nu u_\alpha-\partial_\nu u)+
 \partial_t\theta_\eta\partial_t(\partial_\nu u_\alpha-\partial_\nu u)]=
 \int_0^T\int_\Gamma[\theta_\eta(\partial_\nu u_\alpha-\partial_\nu\widetilde{u}_{\alpha})
 +\partial_t\theta_\eta\partial_t(\partial_\nu u_\alpha-\partial_\nu\widetilde{u}_{\alpha})]
$$

$$+\int_0^T\int_\Gamma g_\eta(x,t)\partial_\nu v_{\eta}(x,t)dt.$$
Now using Proposition \ref{prop3-1}, we get: $$\int_0^T\int_\Gamma[\theta_\eta(\partial_\nu u_\alpha-\partial_\nu u)+
\partial_t\theta_\eta\partial_t(\partial_\nu u_\alpha-\partial_\nu u)]=
\int_0^T\int_\Gamma[\theta_\eta(\partial_\nu u_\alpha-\partial_\nu\widetilde{u}_{\alpha})+
\partial_t\theta_\eta\partial_t(\partial_\nu u_\alpha-\partial_\nu\widetilde{u}_{\alpha})]$$ $$+
\frac{|\eta|^2}{\emph{\textbf{c}}_0}\int_{\Omega^{\prime}}\emph{\textbf{c}}(x)e^{2i\eta.x}dx.$$
To finish, one may use Theorem \ref{thm1} to find that
 $$\int_0^T\int_\Gamma[\theta_\eta(\partial_\nu u_\alpha-\partial_\nu\widetilde{u}_{\alpha})+
\partial_t\theta_\eta\partial_t(\partial_\nu u_\alpha-
\partial_\nu\widetilde{u}_{\alpha})]=O(\alpha^2),$$which achieves the proof.\square
    \end{proof}


We are now in position to describe our identification procedure
which is based on Theorem \ref{main-thm}. Let us neglect the
asymptotically small remainder in the asymptotic formula
(\ref{mainresult}). Then, it follows

\[
\ds \emph{\textbf{c}}(x) \approx \frac{2}{\alpha} \int_{\mathbb{R}^d} \frac{e^{-2 i
\eta \cdot x}}{|\eta|^2} \int_0^T \int_\Gamma \big( \theta_\eta
(\partial_\nu u_\alpha - \partial_\nu u ) +
\partial_t \theta_\eta \partial_t (\partial_\nu u_\alpha- \partial_\nu u ) \big)ds(y)dt, x \in \Omega.
\]
The method of reconstruction we propose here consists in sampling
values of
\[
\ds \frac{1}{|\eta|^2} \int_0^T \int_\Gamma \big( \theta_\eta
(\partial_\nu u_\alpha - \partial_\nu u ) +
\partial_t \theta_\eta \partial_t (\partial_\nu u_\alpha- \partial_\nu u ) \big)ds(y)dt
\]
at some discrete set of points $\eta$ and then calculating the
corresponding inverse Fourier transform. In the following, we will
obtain the more convenient approximation.\\
\begin{corollary}\label{cor1}
Let $\eta \in \mathbb{R}^2$. Suppose that we have all hypothesis of Theorem \ref{main-thm}. Then we have the following approximation \be
\label{emc} \ds \emph{\textbf{c}}(x) \approx - \frac{2}{\alpha} \int_{\mathbb{R}^d}
\frac{e^{-2 i \eta \cdot x}}{|\eta|^2} \int_0^T \int_\Gamma \Bigr(
e^{i |\eta| t}\partial_t ( e^{-i |\eta| t} g_\eta(y,t))
(\partial_\nu u_\alpha - \partial_\nu u)(y, t) \Bigr)ds(y)dt, x \in \Omega,
\ee in terms only of the boundary control
$g_\eta$ which was defined by (\ref{weta}).
\end{corollary}
The desired approximation, given in Corollary \ref{cor1}, may be established by integration by parts over $(0,T)$ for the term \[
\ds \int_0^T \int_\Gamma \partial_t \theta_\eta \partial_t (\partial_\nu u_\alpha- \partial_\nu u )~ds(x)dt,
\]
and by using Remark \ref{rem1}.\\

\section{Conclusion}
We are convinced that the use of approximate formulae such as
(\ref{mainresult}) represents a promising approach to the dynamical
identification and reconstruction of small linear perturbations
in the heat conductivity for the thermal diffusion. We believe that our method yields a significant approximation
to the dynamical identification of small anisotropic cavity $D$, that is embedded in a (homogenous) heat conductive body $\Omega\subset\mathbb{R}^2$ from the measurements of $\partial_\nu u$ on $\Gamma \times (0, T)$.
 The heat conductivity tensor $\emph{\textbf{c}}(x)\in \mathbb{R}^{2\times2}$ is assumed to be symmetric and uniformly positive definite for $x\in D$. The cavity $D$ may be chosen in such a way that the heat conductivity is very low compared with that of the surrounded region $\Omega \backslash \bar{D}$. So, the
problem is a mathematical formulation of a typical inverse problem in thermal imaging. Our method  may be based on appropriate asymptotic expansions combined with an exact control problem to overcome a suitable Fourier transform of the Dirac function representing
a point mass to locate the cavity $D$. This issue will be considered in a forthcoming work.


\end{document}